\documentclass[12pt]{amsart}
\usepackage{amssymb}
\newtheorem{theorem}{Theorem}
\newtheorem{lemma}[theorem]{Lemma}
\newtheorem{proposition}[theorem]{Proposition}

\newtheorem{question}[theorem]{Question}
\newtheorem{definition}[theorem]{Definition}

\begin{document}

\title[ A common extension  ]{ A 
common  extension of Lindel\"of,
H-closed and ccc} 
\author[A. Bella] {Angelo Bella}
\thanks{The research that led to the present paper was partially
supported by a grant of the group GNSAGA of INdAM}
\address{Dipartimento di Matematica e Informatica, viale A. Doria
6, 95125 Catania, Italy}
\email{bella@dmi.unict.it}
\maketitle  

\begin{abstract}
The inequality $|X| \leq 2^{\chi(X)}$ has been proved to be true
for Lindel\"of spaces (Arhangel'ski\u\i, 1969), 
$H$-closed spaces (Dow-Porter, 1982) and ccc spaces (Hajnal-
Ju\'asz 1967), by   quite different
arguments. We present a common extension of all these 
properties which allows us to give a unified proof of
these  three theorems.
\end{abstract}

\bigskip 

All spaces in this  note are assumed to be Hausdorff.

For all undefined notions see Engelking~\cite{Engelking}
 or Porter-Woods~\cite{PW}. 
  $[S]^{<\kappa }$ (resp. $[S]^{\le \kappa }$)  denotes the
collection of all subsets of $S$ of cardinality $<\kappa $ (resp.
$ \le \kappa $).  The
character $\chi(X)$ of a space $X$ is the smallest 
infinite cardinal $\kappa $ such that  every point of $X$ has a
local base of
cardinality not exceeding $\kappa $. 

\smallskip 
The starting point here is the cardinal inequality $|X|\le
2^{\chi(X)}$, proved by Arhangel'ski\u\i~for a  Lindel\"of
space  and by Hajnal and Ju\'asz for a ccc space.

 \emph
{Lindel\"of} means that every open
cover has a countable subcover. \emph{ccc} (= countable chain
condition) means that every family of pairwise disjoint open sets
is countable.

A space $X$ is \emph{H-closed} if every open cover
$\mathcal{V}$
of $X$ has a finite subfamily $\mathcal{W}$ such that
$X=\overline {\bigcup \mathcal W}$.

The previous results are special instances  of  two cardinal
inequalities which were a milestone in the theory of cardinal 
invariants (see \cite{A} and \cite{HJ}).

In  1982 Dow and
Porter \cite{DowPorter82} used H-closed extensions of discrete
spaces to demonstrate that $|X|\leq 2^{\chi(X)}$  for any
H-closed space $X$. In 2006
Hodel~\cite{Hodel2006} gave a proof of the Dow-Porter result
using a closing-off argument. This approach plays a key role
here, as it did in \cite{bella}.  

 Since Lindel\"of, H-closed and ccc are mutually independent
properties in the
class of Hausdorff spaces, 
a  natural general question is the following:
\begin{question} \label{q0}
Is there a property $\mathcal{P}$ of a Hausdorff space $X$ that
a) generalizes H-closed, ccc  and Lindel\"of
simultaneously, and such that b) $|X|\leq 2^{\chi(X)}$ for spaces
$X$ with property $\mathcal{P}$?
\end{question}

A positive answer  to the weaker question concerning only the
couple  Lindel\"of
and H-closed is given in \cite{bella} and \cite{carpor}.
\smallskip 
By combining the results  from \cite{bella} and
\cite{BS}, in this short note we present a property $\mathcal P$
which  provides a
full positive answer to Question \ref{q0}.

 For the proof of some
relevant results given in \cite{bella}, we obviously refer to 
that paper.
  
\medskip  
Following \cite{Hodel2006},
for a given cardinal $\kappa $, a $\kappa $-net in a space $X$ is
a set $\xi=\{x_F:F\in [\kappa ]^{<\omega}\}$.
\begin {definition} (Hodel)  Let $X$ be a space, $A\subseteq X$
and
$\xi=\{x_F:F\in [\kappa ]^{<\omega}\}\subseteq A$. A point $x\in
A$ is a $\theta$-cluster point  of $\xi$ relative to $A$  if,
given any open set $U$ in $X$ with $x\in U$ and any $\alpha
<\kappa $, there exists $F\in [\kappa ]^{<\omega}$ such that
$\alpha \in F$ and $x_F\in \overline { U\cap A}$. If $A=X$, we
say that
$x$ is a $\theta$-cluster point of $\xi$. \end{definition} 

Recall that, given a space $X$ a set $A\subseteq X$ is an H-set 
if for any collection $\mathcal U$ of open sets in $X$, with
$A\subseteq \bigcup \mathcal U$,   there is a finite
subcollection $\mathcal V\subseteq \mathcal U$ such that
$A\subseteq \overline {\bigcup \mathcal V}$.

\begin{lemma} \label{lemma4} [\cite{Hodel2006}, Lemma 4.8]  
Let $X$ be a space and
$\xi=\{x_F:F\in [\chi(X)]^{<\omega}\}$ a $\chi(X)$-net in $X$. If
$\xi $ has a $\theta$-cluster point $x$ in $X$,
then there exists a set $A(\xi)$ such that:

\begin{enumerate}
\item $\xi\subseteq A(\xi)$ and $|A(\xi)|\le \chi(X)$;

\item $x\in A(\xi)$ and $x$  is a $\theta$-cluster point
of $\xi$ relative to $A(\xi)$.
\end{enumerate}
\end{lemma}

The key notion in \cite{bella}  was the following:
\begin{definition}   
Let $X$ be a space and  $A\subseteq X$.  $A$  is said
to
be  $\chi$-net-closed  provided that the following condition
holds:

Given any net $\xi=\{x_F:F\in [\chi(X)]^{<\omega}\}\subseteq A$,
if $\xi$ has a $\theta$-cluster point in $X$, then $\xi$ has  a
$\theta$-cluster point in $A$ relative to $A$. \end{definition} 
\begin{lemma} [\cite{bella}, Lemma 10] \label{lemma5} Let $X$ be
a
space and
$A\subseteq
X$. If $A$
is $\chi$-net-closed, then $A$ is closed in $X$. \end{lemma} 
 
\begin{lemma} [\cite{bella}, Lemma 11] \label{lemma6} If $X$ is
an
H-closed space, then any
$\chi$-net-closed set  is an H-set in $X$. \end{lemma}  

\smallskip 
We present now the property $\mathcal P$ for a possible answer to
Question \ref{q0}. 
\begin{definition} A space $X$ is $\chi$-Lindel\"of  if  for any
$\chi$-
net-closed  set $A\subseteq X$ and any collection $\mathcal 
U=\bigcup \{\mathcal U_\alpha :\alpha <\chi(X)\}$ 
of
open sets in $X$ which covers $A$ there  are countable
subcollections $\mathcal  V_\alpha \subseteq \mathcal  U_\alpha $
for each $\alpha <\chi(X)$ such that
$A\subseteq
\bigcup \{ \overline{\bigcup \mathcal V_\alpha }:\alpha
<\chi(X)\}$.  \end{definition}

By Lemma \ref{lemma5}  we immediately have:
\begin{proposition} \label{pro1}
Every Lindel\"of space is $\chi$-Lindel\"of. 
\end{proposition} 
And by Lemma \ref{lemma6}:
\begin{proposition} \label{pro2}  Every H-closed space is 
$\chi$-Lindel\"of. 
\end{proposition}
The next result requires a proof.
\begin{proposition} Every ccc space $X$ is $\chi$-Lindel\"of.
\end{proposition}
\begin {proof} Let $A$ be a $\chi$-net-closed set and $\mathcal
U=\bigcup\{\mathcal U_\alpha :\alpha <\chi(X)\}$ be a collection
of open sets satisfying $A\subseteq \bigcup\mathcal U$.
For any $\alpha <\chi(X)$  let
$\mathcal   C_\alpha  $ be a maximal collection of pairwise
disjoint
non-empty open sets of $X$ such that  for each $C\in \mathcal
C_ \alpha   $
there
is some $U_C\in \mathcal  U_\alpha  $ with $C\subseteq U_C$.  By
letting $\mathcal  W_\alpha  =\{U_C:C\in \mathcal  C_\alpha  \}$,
the
maximality of  $\mathcal  C_\alpha  $ implies  $\bigcup \mathcal 
U_\alpha 
\subseteq \overline {\bigcup \mathcal  W_\alpha }$ and so
$A\subseteq
\bigcup\{\overline {\bigcup \mathcal  W_\alpha }:\alpha
<\chi(X) \}$.
Since $|\mathcal  W_\alpha  |\le |\mathcal  C_\alpha |\le
\omega $, the space is $\chi$-Lindel\"of. 
\end{proof}
\smallskip 

 The above propositions show that   the property $\mathcal P$  of

being  $\chi$-Lindel\"of is a common extension of Lindel\"of, H-
closed and ccc. 

If $S$ is the Sorgenfrey line and $A([0,1])$ the Aleksandroff
duplicate of the unit interval, then  the space $X=(S\times
S)\oplus A([0,1])$ is first countable and $\chi$-Lindel\"of, but
$X$ is neither Lindel\"of nor H-closed nor ccc. 

 $\chi$-Lindel\"ofness gives a positive answer to Question
\ref{q0}
thanks to the following:  

\begin{theorem} \label{theor}
 If the space $X$ is 
$\chi$-Lindel\"of, then $|X|\le
2^{\chi(X)}$. \end{theorem} 
\begin{proof}  Let $\chi(X)\le \kappa $ and for each $x \in X$
fix a
local base $\mathcal  U_x$ at $x$ satisfying $|\mathcal  U_x|\le
\kappa $.
For any $\kappa $-net $\xi=\{x_F:F\in [\kappa
]^{<\omega}\}\subseteq X$  which has a $\theta$-cluster point in
$X$,  
fix a set $A(\xi)\subseteq X$ satisfying conditions 1) and 2)
 in Lemma \ref{lemma4}. If $\xi$ has no $\theta$-cluster point,
we simply
put $A(\xi)=\xi$.    By transfinite induction, we will construct
a non-
decreasing collection $\{H_\alpha :\alpha <\kappa ^+\}$ of 
subsets of $X$ in such a way that for any $\alpha $ the following
conditions hold:
\begin{enumerate}
\item $|H_\alpha |\le 2^\kappa $;

\item if $\xi=\{x_F:F\in [\kappa ]^{<\omega}\}\subseteq
H_\alpha $, then $A(\xi)\subseteq H_{\alpha +1}$;

\item     if $X\setminus \bigcup\{\overline {\cup
\mathcal  C}:\mathcal  C\in \Gamma\}\ne \emptyset $ for some
$\Gamma\in
\left[[\bigcup\{\mathcal  U_p:p\in H_\alpha \}]^{\le \omega
}\right]^{\le \kappa
}$, then $H_{\alpha +1}\setminus \bigcup\{\overline {\cup\mathcal
C}:\mathcal  C\in \Gamma\}\ne \emptyset $.

\end{enumerate}

Take $x_0\in X$ and fix a choice function $\phi$ on $X$ extended
by letting $\phi(\emptyset )=x_0$. Let $H_0=\{\phi(\emptyset)\}$
and assume to have already defined the sets
$\{H_\beta:\beta<\alpha \}$ according to the previous conditions.
If $\alpha $ is a limit ordinal, then put $H_\alpha = 
{\bigcup\{H_\beta:\beta<\alpha \}}$. If $\alpha =\gamma+1$ we put

\begin{multline*}
H_\alpha = \left\{\phi(X\setminus \bigcup\{\overline {\cup
\mathcal C}:\mathcal  C\in \Gamma\}) : \Gamma\in
\left[[\bigcup\{\mathcal
U_p:p\in
H_\gamma\}]^{\le \omega }\right]^{\le \kappa }\}
\right\}\cup \\
\cup\bigcup\{A(\xi): \xi=\{x_F:F\in [\kappa
]^{<\omega}\}\subseteq H_\gamma\}.
\end{multline*}

A simple counting argument shows that $|H_\alpha |\le 2^\kappa $.

Now, let $H=\bigcup\{H_\alpha :\alpha <\kappa ^+\}$. Since
$\kappa ^+$ is a regular
cardinal and a $\kappa $-net  is a set of cardinality at most
$\kappa $, every $\kappa $-net $\xi\subseteq H$is actually
contained in some $H_\alpha $. Therefore,    Condition 2) in our
inductive
construction ensures that the set $H$ is $\chi$-net-closed.
Since 
$|H|\le 2^\kappa $, if $X=H$ we are done. Assume the contrary
and choose some $q\in X\setminus H$. Let $\mathcal 
U_q=\{U_\alpha 
:  \alpha < \kappa \}$ and let $\mathcal  V_\alpha =\{U_p: p\in H, U_p\in
\mathcal U_p,  U_p\cap
U_\alpha =\emptyset \}$. Since $X$ is $T_2$, we see that the
collection $\mathcal  V=\bigcup\{\mathcal  V_\alpha :\alpha
<\kappa \}$
is a cover
of $H$.  Since $X$ is $\chi$-Lindel\"of,  we may choose
$\mathcal
W_\alpha 
\in
[\mathcal  V_\alpha ]^{\le \omega }$ for every $\alpha  <\kappa 
$ such
that
$H\subseteq \bigcup
\{\overline {\cup\mathcal  W_\alpha }:\alpha <\kappa \}\subseteq
X\setminus \{q\}$. But the regularity
of $\kappa ^+$ implies that there is some $\beta <\kappa ^+$
such that  $\{\mathcal  W_\alpha :\alpha <\kappa \}\subseteq
[\bigcup
\{\mathcal
U_p:p\in H_\beta  \}]^{\le \omega }$. This is in contrast with  
condition 3) in our inductive construction and we reach a
contradiction. \end{proof}

\end{document}